\documentclass{amsart}
\usepackage[all]{xy}
\usepackage{verbatim}
\usepackage{color}
\usepackage{amsthm}
\usepackage{amssymb}
\usepackage[colorlinks=true]{hyperref}
\usepackage{footmisc}
\usepackage{stmaryrd}

\numberwithin{equation}{section}

\newtheorem{theorem}[equation]{Theorem}
\newtheorem*{theorem*}{Theorem} \newtheorem{lemma}[equation]{Lemma}

\newtheorem*{conjecture*}{Mamma Conjecture}
\newtheorem*{conjecture1*}{Mamma Conjecture (revisited)}
\newtheorem{proposition}[equation]{Proposition}
\newtheorem{corollary}[equation]{Corollary}
\newtheorem*{corollary*}{Corollary}

\theoremstyle{remark}

\newtheorem{example}[equation]{Example}

\theoremstyle{remark}
\newtheorem{remark}[equation]{Remark}

\newcommand{\cA}{{\mathcal A}}
\newcommand{\cB}{{\mathcal B}}
\newcommand{\cC}{{\mathcal C}}
\newcommand{\cD}{{\mathcal D}}

\newcommand{\cF}{{\mathcal F}}

\newcommand{\cL}{{\mathcal L}}

\newcommand{\cN}{{\mathcal N}}
\newcommand{\cO}{{\mathcal O}}

\newcommand{\cQ}{{\mathcal Q}}

\newcommand{\cU}{{\mathcal U}}

\newcommand{\cW}{{\mathcal W}}
\newcommand{\cX}{{\mathcal X}}
\newcommand{\cY}{{\mathcal Y}}
\newcommand{\cZ}{{\mathcal Z}}

\newcommand{\bbA}{\mathbb{A}}

\newcommand{\bbF}{\mathbb{F}}

\newcommand{\bbP}{\mathbb{P}}

\newcommand{\bbQ}{\mathbb{Q}}
\newcommand{\bbZ}{\mathbb{Z}}

\DeclareMathOperator{\id}{id}

\DeclareMathOperator{\NChow}{NChow} % category of noncommutative Chow motives
\DeclareMathOperator{\NHom}{NHom} % category of noncommutative Chow motives
\DeclareMathOperator{\NNum}{NNum} % category of noncommutative numerical motives

\DeclareMathOperator{\Num}{Num} % category of numerical motives

\newcommand{\dgcat}{\mathrm{dgcat}} % codimension 
\newcommand{\perf}{\mathrm{perf}}

\newcommand{\Chow}{\mathrm{Chow}}

\newcommand{\dg}{\mathrm{dg}}

\newcommand{\Hom}{\mathrm{Hom}}

\newcommand{\op}{\mathrm{op}}

\newcommand{\too}{\longrightarrow}

\newcommand{\ie}{\textsl{i.e.}\ }

\let\oldmarginpar\marginpar
\def\marginpar#1{\oldmarginpar{\tiny #1}}

\begin{document}

\title[On Grothendieck's standard conjectures in positive characteristic]{On Grothendieck's standard conjectures of type $\mathrm{C}^+$ and $\mathrm{D}$ in positive characteristic}
\author{Gon{\c c}alo~Tabuada}

\address{Gon{\c c}alo Tabuada, Department of Mathematics, MIT, Cambridge, MA 02139, USA}
\email{tabuada@math.mit.edu}
\urladdr{http://math.mit.edu/~tabuada}
\thanks{The author was partially supported by a NSF CAREER Award.}

%\subjclass[2010]{11S40, 14A22, 14C15, 14C25, 14D23, 14M12}
\date{\today}

\abstract{Making use of topological periodic cyclic homology, we extend Grothendieck's standard conjectures of type $\mathrm{C}^+$ and $\mathrm{D}$ (with respect to crystalline cohomology theory) from smooth projective schemes to smooth proper dg categories in the sense of Kontsevich. As a first application, we prove Grothendieck's original conjectures in the new cases of linear sections of determinantal varieties. As a second application, we prove Grothendieck's (generalized) conjectures in the new cases of ``low-dimensional'' orbifolds. Finally, as a third application, we establish a far-reaching noncommutative generalization of Berthelot's cohomological interpretation of the classical zeta function and of Grothendieck's conditional approach to ``half'' of the Riemann hypothesis. Along the way, following Scholze, we prove that the topological periodic cyclic homology of a smooth proper scheme $X$ agrees with the crystalline cohomology theory of $X$ (after inverting the characteristic of the base field).}}

\maketitle
\vskip-\baselineskip
\vskip-\baselineskip
%\vskip-\baselineskip
%\tableofcontents

%\bigskip

%\medskip
%-------------------------------------------------------------------------------
\section{Introduction}
%-------------------------------------------------------------------------------
Let $k$ be a perfect base field of positive characteristic $p>0$, $W(k)$ the associated ring of $p$-typical Witt vectors, and $K:=W(k)[1/p]$ the fraction field of $W(k)$. Given a smooth projective $k$-scheme $X$, let $H^\ast_{\mathrm{crys}}(X):=H^\ast_{\mathrm{crys}}(X/W(k))\otimes_{W(k)} K$ be the crystalline cohomology of $X$, $\pi_X^i$ the $i^{\mathrm{th}}$ K\"unneth projector of $H^\ast_{\mathrm{crys}}(X)$, $Z^\ast(X)_\bbQ$ the $\bbQ$-vector space of algebraic cycles on $X$, and $Z^\ast(X)_\bbQ/_{\!\sim \mathrm{hom}}$ and $Z^\ast(X)_\bbQ/_{\!\sim \mathrm{num}}$ the quotients of $Z^\ast(X)_\bbQ$ with respect to the homological and numerical equivalence relations, respectively. Following Grothendieck \cite{Grothendieck} (see also Kleiman \cite{Kleim1, Kleim}), the standard conjecture of type $\mathrm{C}^+$, denoted by $\mathrm{C}^+(X)$, asserts that the even K\"unneth projector $\pi^+_X:=\sum_i \pi^{2i}_X$ is algebraic, and the standard conjecture of type $\mathrm{D}$, denoted by $\mathrm{D}(X)$, asserts that $Z^\ast(X)_\bbQ/_{\!\sim \mathrm{hom}}=Z^\ast(X)_\bbQ/_{\!\sim \mathrm{num}}$. Both these conjectures hold whenever $\mathrm{dim}(X)\leq 2$. Moreover, the standard conjecture of type $\mathrm{C}^+$ holds for abelian varieties (see Kleiman \cite[2. Appendix]{Kleim}) and also whenever the base field $k$ is finite (see Katz-Messing \cite{KM}). In addition to these cases (and to some other scattered cases), the aforementioned important conjectures remain wide open.

A {\em differential graded (=dg) category} $\cA$ is a category enriched over complexes of $k$-vector spaces; see \S\ref{sub:dg}. Every (dg) $k$-algebra $A$ gives naturally rise to a dg category with a single object. Another source of examples is provided by schemes (or, more generally, by algebraic stacks) since the category of perfect complexes $\perf(X)$ of every quasi-compact quasi-separated $k$-scheme $X$ (or algebraic stack $\cX$) admits a canonical dg enhancement $\perf_\dg(X)$. When $X$ is quasi-projective this dg enhancement is moreover unique; see Lunts-Orlov \cite[Thm. 2.12]{LO}. 

As explained in \S\ref{sec:NC-conj} below, given a smooth proper dg category $\cA$ in the sense of Kontsevich, Grothendieck's standard conjectures of type $\mathrm{C}^+$ and $\mathrm{D}$ admit noncommutative analogues $\mathrm{C}^+_{\mathrm{nc}}(\cA)$ and $\mathrm{D}_{\mathrm{nc}}(\cA)$, respectively.
\begin{theorem}\label{thm:main}
Given a smooth projective $k$-scheme $X$, we have the equivalences of conjectures $\mathrm{C}^+(X)\Leftrightarrow \mathrm{C}^+_{\mathrm{nc}}(\perf_\dg(X))$ and $\mathrm{D}(X)\Leftrightarrow \mathrm{D}_{\mathrm{nc}}(\perf_\dg(X))$.
\end{theorem}
Intuitively speaking, Theorem \ref{thm:main} shows that Grothendieck's standard conjectures of type $\mathrm{C}^+$ and $\mathrm{D}$ belong not only to the realm of algebraic geometry but also to the broad setting of smooth proper dg categories. In what follows, we describe two of the manyfold applications\footnote{For example, Theorem \ref{thm:main} implies immediately that if two smooth projective $k$-schemes $X$ and $Y$ have (Fourier-Mukai) equivalent derived categories, then $\mathrm{C}^+(X)\Leftrightarrow \mathrm{C}^+(Y)$ and $\mathrm{D}(X) \Leftrightarrow \mathrm{D}(Y)$.} of this noncommutative viewpoint; consult also \S\ref{sec:zeta} below for a third application of this noncommutative viewpoint.
%------------------------------------------
\section{Application I: HPD-invariance}
%------------------------------------------
For a survey on Homological Projective Duality (=HPD), we invite the reader to consult Kuznetsov \cite{Kuznetsov-ICM} and/or Thomas \cite{Thomas}. Let $X$ be a smooth projective $k$-scheme equipped with a line bundle $\cL_X(1)$; we write $X \to \bbP(V)$ for the associated morphism, where $V:=H^0(X,\cL_X(1))^\ast$. Assume that the triangulated category $\perf(X)$ admits a Lefschetz decomposition $\langle \bbA_0, \bbA_1(1), \ldots, \bbA_{i-1}(i-1)\rangle$ with respect to $\cL_X(1)$ in the sense of \cite[Def.~4.1]{Kuznetsov-IHES}. Following \cite[Def.~6.1]{Kuznetsov-IHES}, let $Y$ be the HP-dual of $X$, $\cL_Y(1)$ the HP-dual line bundle, and $Y\to \bbP(V^\ast)$ the morphism associated to $\cL_Y(1)$. Given a linear subspace $L\subset V^\ast$, consider the linear sections $X_L:=X\times_{\bbP(V)}\bbP(L^\perp)$ and $Y_L:=Y \times_{\bbP(V^\ast)}\bbP(L)$.
\begin{theorem}[HPD-invariance]\label{thm:HPD}
Let $X$ and $Y$ be as above. Assume that $X_L$ and $Y_L$ are smooth, that $\mathrm{dim}(X_L)=\mathrm{dim}(X)-\mathrm{dim}(L)$ and $\mathrm{dim}(Y_L)=\mathrm{dim}(Y)-\mathrm{dim}(L^\perp)$, and that the conjecture $\mathrm{C}^+_{\mathrm{nc}}(\bbA_0^\dg)$, resp.  $\mathrm{D}_{\mathrm{nc}}(\bbA_0^\dg)$, holds, where $\bbA^{\mathrm{dg}}_0$ stands for the dg enhancement of $\bbA_0$ induced from $\perf_\dg(X)$. Under these assumptions, we have the equivalence $\mathrm{C}^+(X_L) \Leftrightarrow \mathrm{C}^+(Y_L)$, resp. $\mathrm{D}(X_L) \Leftrightarrow \mathrm{D}(Y_L)$.
\end{theorem}
\begin{remark}\label{rk:singular}
The linear section $X_L$ is smooth if and only if $Y_L$ is smooth; see \cite[page~9]{Kuznetsov-ICM}. Moreover, given any {\em general} linear subspace $L \subset V^\ast$, $X_L$ and $Y_L$ are smooth, and $\mathrm{dim}(X_L)=\mathrm{dim}(X)-\mathrm{dim}(L)$ and $\mathrm{dim}(Y_L)=\mathrm{dim}(Y)-\mathrm{dim}(L^\perp)$.
%\item[(ii)] The conjecture $T^l_{\mathrm{nc}}(\bbA_0^{\mathrm{dg}})$ holds, in particular, whenever the triangulated category $\bbA_0$ admits a full exceptional collection.
%\item[(iii)] Theorem \ref{thm:main} holds more generally when $Y$ is singular. In this case we need to replace $Y$ by a noncommutative resolution of singularities $\perf_\dg(Y;\cF)$, where $\cF$ stands for a certain sheaf of noncommutative algebras over $Y$ (see \cite[\S2.4]{Kuznetsov-ICM}). %and conjecture $T^l(Y)$ by its noncommutative analogue $T^l_{\mathrm{nc}}(\perf_\dg(Y;\cF))$. 
\end{remark}
Making use of Theorem \ref{thm:HPD}, we are now able to prove Grothendieck's standard conjectures of type $\mathrm{C}^+$ and $\mathrm{D}$ in new cases. Here is one family of examples:
%-------------------------------------------------------------------------------
\subsection*{Determinantal duality}
%-------------------------------------------------------------------------------
Let $U_1$ and $U_2$ be two $k$-vector spaces of dimensions $d_1$ and $d_2$, respectively, $V:=U_1 \otimes U_2$, and $0 < r< \mathrm{min}\{d_1,d_2\}$ an integer. 

Consider the determinantal variety $\cZ^r_{d_1, d_2}\subset \bbP(V)$ defined as the locus of those matrices $U_2 \to U_1^\ast$ with rank $\leq r$; recall that the condition (rank $\leq r$) can be described as the vanishing of the $(r+1)$-minors of the matrix of indeterminates:
$$ \begin{pmatrix}
x_{1,1} & \cdots & x_{1,d_2} \\
\vdots & \ddots & \vdots \\
x_{d_1, 1} & \cdots & x_{d_1, d_2}
\end{pmatrix}\,.
$$
\begin{example}[Segre varieties]\label{ex:Segre}
In the particular case where $r=1$, the determinantal varieties agree with the classical Segre varieties. Concretely, $\cZ_{d_1, d_2}^1$ agrees with the image of Segre homomorphism $\bbP(U_1) \times \bbP(U_2) \to \bbP(V)$ induced by the map $U_1 \times U_2 \to U_1\otimes U_2$. For example, $\cZ_{2,2}^1$ agrees with the classical quadric hypersurface 
$$\{[x_{1,2}:x_{1,2}: x_{2,1}: x_{2,2}]\,|\,\mathrm{det}\begin{pmatrix} x_{1,1} & x_{1,2} \\ x_{2,1} & x_{2,2} \end{pmatrix}=0 \}\subset \bbP^3\,.$$
% and $\cZ_{3,2}^1$ with the classical Segre threefold  $\{[x]\,|\,\mathrm{rank}(\begin{bmatrix} x_{1,1} & x_{1,2} \\ x_{2,1} & x_{2,2} \\ x_{3,1} & x_{3,2} \end{bmatrix})\leq 1 \}\subset \bbP^5$.
\end{example}
In contrast with the Segre varieties, the determinantal varieties $\cZ^r_{d_1, d_2}, r\geq 2$, are not smooth. The singular locus of $\cZ^r_{d_1, d_2}$ consists of those matrices $U_2 \to U_1^\ast$ with rank $<r$, \ie it agrees with the closed subvariety $\cZ^{r-1}_{d_1, d_2}$. Nevertheless, it is well-known that $\cZ^r_{d_1, d_2}$ admits a canonical Springer resolution of singularities given by the (incidence) projective bundle $\cX^r_{d_1, d_2} := \bbP(U_2 \otimes \cQ) \to \cZ^r_{d_1, d_2}$, where $\cQ$ stands for the tautological quotient vector bundle of the Grassmannian $\mathrm{Gr}(r, U_1)$.
 
Dually, consider the variety $\cW^r_{d_1, d_2}\subset \bbP(V^\ast)$, defined as the locus of those matrices $U^\ast_2 \to U_1$ with corank $\geq r$, and the associated Springer resolution of singularities $\cY^r_{d_1, d_2}:=\bbP(U_2^\ast \otimes \cU^\ast) \to \cW^r_{d_1, d_2}$, where $\cU$ stands for the tautological subvector bundle of $\mathrm{Gr}(r,U_1)$. As explained in \cite[\S1]{Tate}, work of Bernardara-Bolognesi-Faenzi \cite{BBF} and Buchweitz-Leuschke-Van den Bergh \cite{BLV} implies that $X:=\cX^r_{d_1, d_2}$ and $Y:=\cY^r_{d_1, d_2}$ are HP-dual to each other with respect to a certain Lefschetz decomposition $\perf(X)=\langle \bbA_0, \bbA_1(1), \ldots, \bbA_{d_2 r -1}(d_2r -1)\rangle$. Moreover, the dg category $\bbA_0^{\mathrm{dg}}$ is Morita equivalent to a finite-dimensional $k$-algebra of finite global dimension $A$; consult the proof of \cite[Prop.~1.5]{Tate}. Thanks to Proposition \ref{prop:fin} below, this implies that the conjectures $\mathrm{C}^+_{\mathrm{nc}}(\bbA_0^{\mathrm{dg}})$ and $\mathrm{D}_{\mathrm{nc}}(\bbA_0^{\mathrm{dg}})$ hold. Consequently, by combining Theorem \ref{thm:HPD} with Remark \ref{rk:singular}, we obtain the following result: 
%
%
% Recall that the determinantal varieties with $r=1$ are the classical Segre varieties. For example, $\cZ^1_{2,2}\subset \bbP^3$ is the quadric surface defined as the zero locus of the $2\times 2$ minor $v_0v_3 - v_1 v_2$. It is well-known that $\cZ^r_{d_1, d_2}$ admits a canonical Springer resolution of singularities $\cX^r_{d_1, d_2} \to \cZ^r_{d_1, d_2}$, which comes equipped with a projection $q\colon \cX^r_{d_1, d_2} \to \mathrm{Gr}(r, U_1)$ to the Grassmannian of $r$-dimensional subspaces in $U_1$. Following \cite[\S3.3]{Marcello}, the category $\perf(X)$, with $X:=\cX^r_{d_1, d_2}$, admits a Lefschetz decomposition $\langle \bbA_0, \bbA_1(1), \ldots, \bbA_{d_2 r - 1}(d_2 r -1)\rangle$, where $\bbA_0=\bbA_1= \cdots = \bbA_{d_2 r -1}= q^\ast(\perf(\mathrm{Gr}(r,U_1)))\simeq \perf(\mathrm{Gr}(r,U_1))$. 
\begin{corollary}\label{cor:last}
Given any general linear subspace $L \subset V^\ast$, we have the equivalences of equivalences $\mathrm{C}^+(X_L) \Leftrightarrow \mathrm{C}^+(Y_L)$ and $\mathrm{D}(X_L) \Leftrightarrow \mathrm{D}(Y_L)$.
\end{corollary}
%Dually, consider the variety $\cW^r_{d_1, d_2}\subset \bbP(V^\ast)$, defined as the locus of those matrices $U^\ast_2 \to U_1$ with corank $\geq r$, and the associated Springer resolutions of singularities $Y:=\cY^r_{d_1, d_2} \to \cW^r_{d_1, d_2}$. As proved in Bernardara-Bolognesi-Faenzi in\footnote{In \cite[Prop.~3.4 and Thm.~3.5]{Marcello} the authors worked over an algebraically closed field of characteristic zero. However, the same proof holds {\em mutatis mutandis} over $k=\bbF_q$. Simply replace the reference \cite{Kapranov} to Kapranov's full strong exceptional collection on $\perf(\mathrm{Gr}(r, U_1))$ by the reference \cite[Thm.~1.3]{VdB} to Buchweitz-Leuschke-Van den Bergh's tilting bundle on $\perf(\mathrm{Gr}(r, U_1))$. I am grateful to Marcello Bernardara for discussions concerning this issue.} \cite[Prop.~3.4 and Thm.~3.5]{Marcello}, $X$ and $Y$ are HP-dual to each other. Given a general linear subspace $L \subseteq V^\ast$, consider the smooth linear sections $X_L$ and $Y_L$. Note that whenever $\bbP(L^\perp)$ does not intersects the singular locus of $\cZ^r_{d_1, d_2}$, \ie the locus of those matrices $U_2 \to U_1^\ast$ with rank $<r$, we have $X_L=\bbP(L^\perp) \cap \cZ^r_{d_1, d_2}$.\begin{corollary}\label{cor:last}
%We have the equivalence $T^l(X_L)\Leftrightarrow T^l(Y_L)$.
%\end{corollary}
By construction, $\mathrm{dim}(X)= r(d_1+ d_2 -r)-1$ and $\mathrm{dim}(Y)= r(d_1-d_2-r) + d_1 d_2 -1$. Consequently, the associated linear sections have the following dimensions:
$$
\mathrm{dim}(X_L) =  r(d_1+d_2-r) -1 - \mathrm{dim}(L) \quad 
\mathrm{dim}(Y_L) = r(d_1 - d_2 - r) -1 +\mathrm{dim}(L)\,. 
$$
Since Grothendieck's standard conjectures of type $\mathrm{C}^+$ and $\mathrm{D}$ hold in dimensions $\leq 2$, we hence obtain from Corollary \ref{cor:last} the following result(s):
\begin{theorem}[Linear sections of determinantal varieties]\label{thm:last}
Let $X_L$ and $Y_L$ be smooth linear sections of determinantal varieties as in Corollary \ref{cor:last}.
\begin{itemize}
\item[(i)] When $r(d_1 + d_2 -r) -1 - \mathrm{dim}(L)\leq 2$, the conjectures $\mathrm{C}^+(Y_L)$~and~$\mathrm{D}(Y_L)$~hold.
\item[(ii)] When $r(d_1 -d_2 -r) -1 + \mathrm{dim}(L)\leq 2$, the conjectures $\mathrm{C}^+(X_L)$~and~$\mathrm{D}(X_L)$~hold.
\end{itemize}
\end{theorem}
\begin{corollary}[Square matrices]\label{cor:HPD}
Let $d_1=d_2$. Given any general linear subspace $L \subset V^\ast$ of dimension $r^2+i$, $1\leq i \leq3 $, the conjectures $\mathrm{C}^+(X_L)$ and $\mathrm{D}(X_L)$ hold.
\end{corollary}
To the best of the author's knowledge, Theorem \ref{thm:last} (and Corollary \ref{cor:HPD}) is new in the literature. In particular, it proves Grothendieck's standard conjecture of type $\mathrm{D}$ in several new cases; consult Remarks \ref{rk:typeB}-\ref{rk:zero} below. For example, note that in Corollary \ref{cor:HPD} the linear section $X_L$ is of dimension $r((2d-r)-r)-1 -i, 1\leq i \leq 3$, with $d:=d_1=d_2$. Therefore, by letting $d \to \infty$, we obtain infinitely many new examples of smooth projective $k$-schemes $X_L$, of arbitrary high dimension, satisfying Grothendieck's standard conjecture of type $\mathrm{D}$.
\begin{remark}[Standard conjecture of type $\mathrm{B}$]\label{rk:typeB}
Recall from \cite{Grothendieck} the definition of Grothendieck's standard conjecture of type $\mathrm{B}$ (a.k.a. the standard conjecture of Lefschetz type). This conjecture holds for the projective bundles $X$ and $Y$, is stable under hyperplane sections, and implies the standard conjecture of type $\mathrm{C}^+$; see \cite[\S4]{Kleim1}. Consequently, the standard conjecture of Lefschetz type yields an alternative ``geometric'' proof of Theorem \ref{thm:last} for the standard conjecture of~type~$\mathrm{C}^+$.
\end{remark}
%It should be emphasized that since the base field $k$ is of positive characteristic, it is {\em not} known if the standard conjecture of type $B$ implies the standard conjecture of type $\mathrm{D}$! If we further assume the standard conjecture of type $I$ (a.k.a. the standard conjecture of Hodge type), then the implication $B \Rightarrow D$ holds. However, in positive characteristic, the standard conjecture of type $I$ is only known in dimension $\leq 3$.
%\end{remark}
\begin{remark}[Standard conjecture of type $\mathrm{I}$]\label{rk:zero}
Recall from \cite{Grothendieck} the definition of Grothendieck's standard conjecture of type $\mathrm{I}$ (a.k.a. the standard conjecture of Hodge type). As explained in {\em loc. cit.}, given any smooth projective $k$-scheme $X$, we have the implication $\mathrm{B}(X) + \mathrm{I}(X) \Rightarrow \mathrm{D}(X)$. On the one hand, when the base field $k$ is of characteristic zero\footnote{The characteristic zero analogue of Theorem \ref{thm:last} was proved in \cite[Thm.~1.11]{CD}.}, the conjecture $\mathrm{I}(X)$ holds (thanks to the Hodge index theorem). On the other hand, in positive characteristic, the conjecture $\mathrm{I}(X)$ is only known to hold when $\mathrm{dim}(X)\leq 3$. Consequently, in contrast with Remark \ref{rk:typeB}, the standard conjecture of Lefschetz type does {\em not} yields an alternative ``geometric'' proof of Theorem \ref{thm:last} for the standard conjecture of type $\mathrm{D}$.
\end{remark}
Finally, note that whenever $\bbP(L^\perp) \subset \bbP(V)$ does not intersects the singular locus of $\cZ^r_{d_1, d_2}$, we have $X_L=\bbP(L^\perp) \cap \cZ^r_{d_1, d_2}$. In other words, $X_L$ is a linear section of a determinantal variety. Here are some examples:
\begin{example}[Segre varieties]\label{ex:Segre1}
Let $r=1$. In this case, as mentioned in Example \ref{ex:Segre}, the determinantal variety $\cZ^1_{d_1, d_2}\subset \bbP^{d_1 d_2 -1}$ agrees with the smooth Segre variety of  dimension $d_1 + d_2 -2$. Therefore, thanks to Theorem \ref{thm:last}(ii), given any general linear subspace $L \subset V^\ast$ of dimension $(d_2 -d_1) +i, 2 \leq i \leq 4$, the associated smooth linear section $X_L\subset \cZ^1_{d_1, d_2}$ has dimension $2d_1 -2 - i, 2 \leq i \leq 4$, and satisfies Grothendieck's standard conjectures of type $\mathrm{C}^+$ and $\mathrm{D}$. %Note that in all these cases, $X_L$ is a linear section of the Segre variety $\cZ^1_{d_1, d_2}$.
\end{example}
\begin{example}[Rational normal scrolls]
Let $r=1$ and $d_2=2$. In this case, the Segre variety $\cZ^2_{d_1, 2} \subset \bbP^{2d_1-1}$ agrees with the rational normal $d_1$-fold scroll $S_{1, \ldots, 1}$; see \cite[Ex.~8.27]{Harris}. Take $d_1=4$, resp. $d_1=5$, and choose a linear subspace $L \subset V^\ast$ of dimension $1$ for which the associated hyperplane $\bbP(L^\perp)\subset \bbP^7$, resp. $\bbP(L^\perp) \subset \bbP^9$, does not contains any $3$-plane, resp. $4$-plane, of the rulling of $S_{1,1,1,1}$, resp. $S_{1,1,1,1,1}$. Note that this is a general condition on $L$. By combining Example \ref{ex:Segre1} with \cite[Prop.~2.5]{Faenzi}, we hence conclude that the rational normal $3$-fold scroll $X_L=S_{1,1,2}$, resp. $4$-fold scroll $X_L=S_{1,1,1,2}$, satisfies Grothendieck's standard conjectures of type $\mathrm{C}^+$ and $\mathrm{D}$.
\end{example}
\begin{example}[Square matrices]
Let $d_1=d_2=4$ and $r=2$. In this case, the determinantal variety $\cZ^2_{4,4}\subset \bbP^{15}$ has dimension $11$ and its singular locus is the $6$-dimensional Segre variety $\cZ^1_{4,4}$. Given any general linear subspace $L\subset V^\ast$ of dimension $7$, the associated smooth linear section $X_L$ is $4$-dimensional and, thanks to Corollary \ref{cor:HPD}, it satisfies Grothendieck's standard conjectures of type $\mathrm{C}^+$ and $\mathrm{D}$. Note that since $\mathrm{codim}(L^\perp)=7>6 = \mathrm{dim}(\cZ^1_{4,4})$, the subspace $\bbP(L^\perp)\subset \bbP^{15}$ does {\em not} intersects the singular locus $\cZ^1_{4,4}$ of $\cZ^2_{4,4}$. Therefore, in all these cases, the $4$-fold $X_L$ is a linear section of the determinantal variety $\cZ^2_{4,4}$.
\end{example}

%-------------------------------------------------------------------------------
\section{Application II: Grothendieck's standard conjectures for orbifolds}
%-------------------------------------------------------------------------------
Theorem \ref{thm:main} allows us to easily extend Grothendieck's standard conjectures of type $\mathrm{C}^+$ and $\mathrm{D}$ from smooth projective $k$-schemes $X$ to smooth proper algebraic $k$-stacks $\cX$ by setting $\mathrm{C}^+(\cX):=\mathrm{C}^+_{\mathrm{nc}}(\perf_\dg(\cX))$ and $\mathrm{D}(\cX):=\mathrm{D}_{\mathrm{nc}}(\perf_\dg(\cX))$.
\begin{theorem}[Orbifolds]\label{thm:orbifold}
Let $G$ be a finite group of order $n$, $X$ a smooth projective $k$-scheme equipped with a $G$-action, and $\cX:=[X/G]$ the associated global orbifold. If $p \nmid n$, then we have the following implications of conjectures
\begin{eqnarray*}
\sum_{\sigma \subseteq G} \mathrm{C}^+(X^\sigma) \Rightarrow  \mathrm{C}^+(\cX) && \sum_{\sigma \subseteq G} \mathrm{D}(X^\sigma \times \mathrm{Spec}(k[\sigma])) \Rightarrow  \mathrm{D}(\cX)\,,
\end{eqnarray*}
where $\sigma$ is a cyclic subgroup of $G$. Moreover, whenever $k$ contains the $n^{\mathrm{th}}$ roots of unity, the conjecture $\mathrm{D}(X^\sigma \times \mathrm{Spec}(k[\sigma]))$ can be replaced by $\mathrm{D}(X^\sigma)$.
\end{theorem}
Theorem \ref{thm:orbifold} leads automatically to a proof of Grothendieck's standard conjectures of type $\mathrm{C}^+$ and $\mathrm{D}$ for (global) orbifolds in the following cases:
\begin{corollary}\label{cor:examples} Assume that $p\nmid n$ and let $\cX:=[X/G]$ be as in Theorem \ref{thm:orbifold}.
\begin{itemize}
\item[(i)] If the base field $k$ is finite, then the conjecture $\mathrm{C}^+(\cX)$ holds.
\item[(ii)] If $\mathrm{dim}(X)\leq 2$, if $\mathrm{dim}(X)=3$ and $\mathrm{C}^+(X)$ holds, or if $X$ is an abelian variety and $G$ acts by group homomorphisms, then the conjecture $\mathrm{C}^+(\cX)$ holds.
\item[(iii)] If $\mathrm{dim}(X)\leq 2$, then the conjecture $\mathrm{D}(\cX)$ holds.
\item[(iv)] Assume moreover that $k$ contains the $n^{\mathrm{th}}$ roots of unity. If $\mathrm{dim}(X)=3$ and $\mathrm{D}(X)$ holds, then the conjecture $\mathrm{D}(\cX)$ holds.
\end{itemize}
\end{corollary}
%-------------------------------------------------------------------------------
\section{Preliminaries}
%-------------------------------------------------------------------------------
%Throughout the article, $k$ denotes a perfect base field of characteristic $p>0$.
\subsection{Dg categories}\label{sub:dg} 
For a survey on dg categories, we invite the reader to consult Keller's ICM address \cite{ICM-Keller}. Let $(\cC(k),\otimes, k)$ be the category of dg $k$-vector spaces. A {\em differential
  graded (=dg) category $\cA$} is a category enriched over $\cC(k)$
and a {\em dg functor} $F\colon\cA\to \cB$ is a functor enriched over
$\cC(k)$. In what follows, we write $\dgcat(k)$ for the category of (essentially small) dg categories and dg functors.

Let $\cA$ be a dg category. The opposite dg category $\cA^\op$ has the
same objects and $\cA^\op(x,y):=\cA(y,x)$. A {\em right dg
  $\cA$-module} is a dg functor $\cA^\op \to \cC_\dg(k)$ with values
in the dg category $\cC_\dg(k)$ of dg $k$-vector spaces. Following \cite[\S3.2]{ICM-Keller}, the {\em derived category} $\cD(\cA)$ of $\cA$ is defined as the localization of the category of right dg $\cA$-modules with respect to the objectwise quasi-isomorphisms. In what follows, we write $\cD_c(\cA)$ for the triangulated subcategory of compact objects. A dg functor $F\colon\cA\to \cB$ is called a {\em Morita equivalence} if it
induces an equivalence on derived categories $\cD(\cA) \simeq 
\cD(\cB)$; see \cite[\S4.6]{ICM-Keller}. The {\em tensor product $\cA\otimes\cB$} of dg categories is defined
as follows: the set of objects is $\mathrm{obj}(\cA)\times \mathrm{obj}(\cB)$ and
$(\cA\otimes\cB)((x,w),(y,z)):= \cA(x,y) \otimes \cB(w,z)$. A {\em dg $\cA\text{-}\cB$-bimodule} is a dg functor
$\cA\otimes \cB^\op \to \cC_\dg(k)$. An example is the dg
$\cA\text{-}\cB$-bimodule ${}_F\cB:\cA\otimes \cB^\op \to \cC_\dg(k), (x,z) \mapsto \cB(z,F(x))$ associated to a dg functor $F:\cA\to \cB$. Following Kontsevich \cite{Miami,finMot,IAS}, a dg category $\cA$ is called {\em smooth} if the dg $\cA\text{-}\cA$-bimodule ${}_{\id}\cA$ belongs to $\cD_c(\cA^\op\otimes \cA)$ and {\em proper} if $\sum_i \mathrm{dim}\, H^i\cA(x,y)< \infty$ for any ordered pair of objects $(x,y)$. Examples include finite-dimensional $k$-algebras of finite global dimension $A$ as well as the dg categories of perfect complexes $\perf_\dg(X)$ associated to smooth proper $k$-schemes $X$. In what follows, we write $\dgcat_{\mathrm{sp}}(k)$ for the full subcategory of smooth proper dg categories.

\subsection{Orbit categories}\label{sub:orbit}
Let $(\cC,\otimes, {\bf 1})$ be a $\bbQ$-linear, additive, symmetric monoidal category and $\cO \in \cC$ a $\otimes$-invertible object. The associated {\em orbit category} $\cC/_{\!-\otimes \cO}$ has the same objects as $\cC$ and morphisms 
$$\Hom_{\cC/_{\!-\otimes \cO}}(a,b):=\bigoplus_{n \in \bbZ} \Hom_\cC(a, b \otimes \cO^{\otimes n})\,.$$ 
Given objects $a, b, c$ and composable morphisms $\mathrm{f}=\{f_n\}_{n \in \bbZ}$ and $\mathrm{g}=\{g_n\}_{n \in \bbZ}$, the $i^{\mathrm{th}}$-component of $\mathrm{g}\circ \mathrm{f}$ is defined as $\sum_n (g_{i -n} \otimes \cO^{\otimes n})\circ f_n$. The canonical functor 
\begin{eqnarray*}
\iota\colon \cC \to \cC/_{\!-\otimes \cO} & a \mapsto a & f \mapsto \mathrm{f}=\{f_n\}_{n \in \bbZ}\,,
\end{eqnarray*}
where $f_0=f$ and $f_n=0$ if $n\neq 0$, is endowed with an isomorphism $\iota \circ (-\otimes \cO) \Rightarrow \iota$ and is $2$-universal among all such functors. Moreover, the category $\cC/_{\!-\otimes \cO}$ is $\bbQ$-linear, additive, and inherits from $\cC$ a symmetric monoidal structure making $\iota$ into a symmetric monoidal functor.
%-------------------------------------------------------------------------------
\section{Topological periodic cyclic homology}
%-------------------------------------------------------------------------------
Thanks to the work of Hesselholt \cite[\S4]{Hesselholt} and Blumberg-Mandell\footnote{See also the work of Antieau-Mathew-Nikolaus \cite{AMN}.} \cite[Thm.~A]{BM}, topological periodic cyclic homology\footnote{Recall that topological periodic cyclic homology is defined as the Tate cohomology of the circle group acting on topological Hochschild homology.} yields a symmetric monoidal functor
\begin{equation}\label{eq:TP}
TP_\pm(-)_{1/p}\colon \dgcat(k) \too \mathrm{vect}_{\bbZ/2}(K)
\end{equation} 
with values in the category of finite-dimensional $\bbZ/2$-graded $K$-vector spaces; consult also \cite[\S4]{Positive}. The following result, which is of independent interest, will be used below in the proof of Theorem \ref{thm:main} and Corollary \ref{cor:zeta}.
\begin{theorem}[Scholze \cite{Scholze}]\label{thm:TP}
Given a smooth proper $k$-scheme $X$, we have a natural isomorphism of $\bbZ/2$-graded $K$-vector spaces:
\begin{equation}\label{eq:Scholze}
TP_\pm(\perf_\dg(X))_{1/p} \simeq (\bigoplus_{i\,\mathrm{even}}H^i_{\mathrm{crys}}(X), \bigoplus_{i\,\mathrm{odd}}H^i_{\mathrm{crys}}(X))\,.
\end{equation}
\end{theorem}
\begin{proof}
In order to simplify the exposition, we will write $TP(X)$ for the spectrum $TP(\perf_\dg(X))$. Following Bhatt-Morrow-Scholze \cite[\S9.4]{BMS2}, let us choose a prime number $l\neq p$ and consider the associated Adams operation $\psi_l$; since we are working over a perfect base field $k$ of characteristic $p>0$, the spectrum $TP(X)$ is already $p$-complete. As proved in \cite[Thm.~1.12(2)]{BMS2}, the spectrum $TP(X)$ admits a ``motivic'' exhaustive decreasing $\bbZ$-indexed filtration $\{\mathrm{fil}^n TP(X)\}_{n \in \bbZ}$. After inverting $p$, this leads to an induced filtration $\{\mathrm{fil}^n TP(X)[1/p]\}_{n \in \bbZ}$ of $TP(X)[1/p]$. Since the Adams operation $\psi_l$ preserves this filtration, we hence obtain the $K$-linear homomorphisms
\begin{equation}\label{eq:homo1}
(\pi_\ast(\mathrm{fil}^n TP(X)[1/p]))^{\psi_l=l^n} \too (\pi_\ast(TP(X)[1/p]))^{\psi_l=l^n} \simeq TP_\ast(X)_{1/p}^{\psi_l=l^n}
\end{equation}
\begin{equation}\label{eq:homo2}
(\pi_\ast(\mathrm{fil}^n TP(X)[1/p]))^{\psi_l=l^n} \too (\pi_\ast(\mathrm{gr}^nTP(X)[1/p]))^{\psi_l=l^n}\,,
\end{equation}
where $(-)^{\psi_l=l^n}$ stands for the $K$-linear subspace of those elements $v$ such that $\psi_l(v)=l^n \cdot v$. We claim that the above homomorphisms \eqref{eq:homo1}-\eqref{eq:homo2} are invertible. Consider the following cofiber sequence of spectra
\begin{equation}\label{eq:cofiber}
\mathrm{fil}^n TP(X)[1/p] \too TP(X)[1/p] \too \frac{TP(X)[1/p]}{\mathrm{fil}^n TP(X)[1/p]}=:\mathrm{cofiber}
\end{equation}
and the endomorphism of the associated long exact sequence of $K$-vector spaces:
$$
\xymatrix@C=1.8em@R=2.5em{
\cdots \pi_{\ast+1}(\mathrm{cofiber}) \ar[d]_-{\psi^l-l^n}\ar[r] & \pi_\ast(\mathrm{fil}^n TP(X)[1/p]) \ar[d]_-{\psi^l-l^n} \ar[r] & \pi_\ast(TP(X)[1/p])\ar[d]_-{\psi^l-l^n} \ar[r] & \pi_\ast(\mathrm{cofiber}) \ar[d]_-{\psi^l-l^n} \cdots \\
\cdots \pi_{\ast+1}(\mathrm{cofiber}) \ar[r] & \pi_\ast(\mathrm{fil}^n TP(X)[1/p]) \ar[r] & \pi_\ast(TP(X)[1/p]) \ar[r] & \pi_\ast(\mathrm{cofiber}) \cdots
}
$$
Since $X$ is a smooth proper $k$-scheme, the dg category $\perf_\dg(X)$ is smooth and proper. This implies that the $K$-vector spaces $\pi_\ast(TP(X)[1/p])\simeq TP_\ast(X)_{1/p}$ (and hence $\pi_\ast(\mathrm{fil}^n TP(X)[1/p])$) are finite dimensional. Therefore, thanks to the general Lemma \ref{lem:aux2} below, in order to prove that the above homomorphisms \eqref{eq:homo1} are invertible, it suffices to show that the following endomorphisms are invertible:
\begin{equation}\label{eq:endomorphism}
\psi^l - l^n \colon \pi_\ast(\mathrm{cofiber}) \too \pi_\ast(\mathrm{cofiber})\,.
\end{equation}
Note that the spectrum $\frac{TP(X)[1/p]}{\mathrm{fil}^n TP(X)[1/p]}$ comes naturally equipped with the following exhaustive decreasing filtration $\{\frac{TP(X)[1/p]}{\mathrm{fil}^m TP(X)[1/p]}\}_{m<n}$, whose graded pieces are equal to $\{\mathrm{gr}^m TP(X)[1/p]\}_{m<n}$. As proved in \cite[Prop.~9.14]{BMS2}, the induced endomorphism $\psi^l$ of $\pi_\ast(\mathrm{gr}^m TP(X)[1/p])$ acts by multiplication with $l^m$. Since $m<n$ and $K$ is of characteristic zero, this implies that the following endomorphisms are invertible:
\begin{eqnarray}\label{eq:invertible}
& \psi^l - l^n = l^m - l^n \colon \pi_\ast(\mathrm{gr}^m TP(X)[1/p]) \stackrel{\simeq}{\too} \pi_\ast(\mathrm{gr}^m TP(X)[1/p]) & m<n\,.
\end{eqnarray}
Now, a standard inductive argument using the isomorphisms \eqref{eq:invertible} and the $5$-lemma allows us to conclude that the above homomorphisms \eqref{eq:homo1} are invertible. The proof of the invertibility of the homomorphisms \eqref{eq:homo2} is similar: simply replace the above cofiber sequence \eqref{eq:cofiber} by the following fiber sequence of spectra:
$$ \mathrm{fil}^{n+1} TP(X)[1/p] \too \mathrm{fil}^n TP(X)[1/p] \too \mathrm{gr}^n TP(X)[1/p]\,.$$
This concludes the proof of our claim.

As mentioned above, the induced endomorphism $\psi^l$ of $\pi_\ast(\mathrm{gr}^n TP(X)[1/p])$ acts by multiplication with $l^n$. Thanks to \eqref{eq:homo1}-\eqref{eq:homo2}, this leads to natural isomorphisms:
$$ TP_\ast(X)_{1/p}^{\psi_l=l^n}\simeq (\pi_\ast(\mathrm{gr}^n TP(X)[1/p]))^{\psi_l=l^n} = \pi_\ast(\mathrm{gr}^n TP(X)[1/p])\,.$$
Hence, using the fact that the filtration $\{\mathrm{fil}^n TP(X)[1/p]\}_{n \in \bbZ}$ of $TP(X)[1/p]$ is exhaustive, we obtain the following natural isomorphisms of $K$-vector spaces:
\begin{equation}\label{eq:natural}
TP_\ast(X)_{1/p} \simeq \bigoplus_{n \in \bbZ} TP_\ast(X)_{1/p}^{\psi_l=l^n} \simeq \bigoplus_{n \in \bbZ} \pi_\ast(\mathrm{gr}^n TP(X)[1/p])\,.
\end{equation}
Making use of the natural isomorphisms $\pi_\ast(\mathrm{gr}^n TP(X)[1/p])\simeq H_{\mathrm{crys}}^{\ast- 2n}(X)$ constructed by Bhatt-Morrow-Scholze in \cite[Thms. 1.10 and 1.12(4)]{BMS2}, we hence conclude that $TP_\ast(X)_{1/p}\simeq \bigoplus_{n \in \bbZ} H^{\ast-2n}_{\mathrm{crys}}(X)$. This automatically yields the natural isomorphism of $\bbZ/2$-graded $K$-vector spaces \eqref{eq:Scholze}, and so the proof is finished.
\end{proof}
\begin{lemma}\label{lem:aux2}
Consider the following commutative diagram of vector spaces:
$$
\xymatrix{
V_1 \ar[d]_-{f_1} \ar[r] & V_2 \ar[d]_-{f_2} \ar[r] & V_3 \ar[d]_-{f_3} \ar[r] & V_4 \ar[d]_-{f_4} \\
V_1 \ar[r] & V_2 \ar[r] & V_3 \ar[r] & V_4\,.
}
$$
Assume that the rows are exact, that $V_2$ and $V_3$ are finite dimensional, that $f_1$ is surjective, and that $f_4$ is injective. Under these assumptions, the induced homomorphism $\mathrm{Ker}(f_2) \to \mathrm{Ker}(f_3)$ is invertible.
\end{lemma}
\begin{proof}
On the one hand, a simple diagram chasing argument implies that the induced homomorphism $\mathrm{Ker}(f_2) \to \mathrm{Ker}(f_3)$ is surjective. On the other hand, a dual diagram chasing argument implies that the induced homomorphism $\mathrm{coKer}(f_2) \to \mathrm{coKer}(f_3)$ is injective. Making use of the equalities $\mathrm{dim}\,\mathrm{Ker}(f_2)= \mathrm{dim}\,\mathrm{coKer}(f_2)$ and $\mathrm{dim}\,\mathrm{Ker}(f_3)= \mathrm{dim}\,\mathrm{coKer}(f_3)$ and of the finite dimensionality of the vector spaces $\mathrm{Ker}(f_2)$ and $\mathrm{Ker}(f_3)$, we hence conclude that the induced homomorphism $\mathrm{Ker}(f_2) \to \mathrm{Ker}(f_3)$ is moreover injective.
\end{proof}
%-------------------------------------------------------------------------------
\section{Noncommutative motives}\label{sec:NC}
%-------------------------------------------------------------------------------
For a book, resp. survey, on noncommutative motives, we invite the reader to consult \cite{book}, resp. \cite{survey}. Recall from \cite[\S4.1]{book} the definition of the category of noncommutative Chow motives $\NChow(k)_\bbQ$. By construction, this $\bbQ$-linear category is additive, rigid symmetric monoidal, and comes equipped with a symmetric monoidal functor $U(-)_\bbQ\colon \dgcat_{\mathrm{sp}}(k) \to \NChow(k)_\bbQ$. Moreover, we have
$$\Hom_{\NChow(k)_\bbQ}(U(\cA)_\bbQ, U(\cB)_\bbQ)\simeq K_0(\cD_c(\cA^\op \otimes \cB))_\bbQ =:K_0(\cA^\op \otimes \cB)_\bbQ\,.$$
Recall from \cite[Prop.~4.2]{Positive} that the above functor \eqref{eq:TP} yields a $\bbQ$-linear symmetric monoidal functor $TP_\pm(-)_{1/p}\colon \NChow(k)_\bbQ \to \mathrm{vect}_{\bbZ/2}(K)$. Under these notations, the category of {\em noncommutative homological motives} $\NHom(k)_\bbQ$ is defined as the idempotent completion of the quotient $\NChow(k)_\bbQ/\mathrm{Ker}(TP_\pm(-)_{1/p})$. 

Given a $\bbQ$-linear, additive, rigid symmetric monoidal category $(\cC,\otimes, {\bf 1})$, its $\cN$-ideal is defined as follows ($\mathrm{tr}(g\circ f)$ stands for the categorical trace of $g\circ f$):
\begin{equation}\label{eq:N}
\cN(a, b) := \{f \in \Hom_\cC(a, b) \,\,|\,\, \forall g \in \Hom_\cC(b, a)\,\, \mathrm{we}\,\,\mathrm{have}\,\, \mathrm{tr}(g\circ f) =0 \}\,.
\end{equation}
This is the largest $\otimes$-ideal of $\cC$ distinct from the entire category. Under these notations, the category of {\em noncommutative numerical motives} $\NNum(k)_\bbQ$ is defined as the idempotent completion of the quotient $\NChow(k)_\bbQ/\cN$. 

%-------------------------------------------------------------------------------
\section{Noncommutative standard conjectures of type $\mathrm{C}^+$ and $\mathrm{D}$}\label{sec:NC-conj}
%-------------------------------------------------------------------------------
Given a smooth proper dg category $\cA$, consider the even K\"unneth projector $\pi^\cA_+$ of the $\bbZ/2$-graded $K$-vector $TP_\pm(\cA)_{1/p}$, as well as the following $\bbQ$-vector spaces\footnote{As explained in \cite[\S6]{Positive}, the $\bbQ$-vector space $K_0(\cA)_\bbQ/_{\!\mathrm{num}}$ can be alternatively defined as the $\bbQ$-linearization of the quotient of $K_0(\cA)$ by the (left=right) kernel of the classical Euler pairing.}:
\begin{eqnarray*}
K_0(\cA)_\bbQ/_{\!\sim \mathrm{hom}}:= \Hom_{\NHom(k)_\bbQ}(U(k)_\bbQ, U(\cA)_\bbQ)\\ 
K_0(\cA)_\bbQ/_{\!\sim \mathrm{num}}:= \Hom_{\NNum(k)_\bbQ}(U(k)_\bbQ, U(\cA)_\bbQ) \,.
\end{eqnarray*}
Under these notations, Grothendieck's standard conjectures of type $\mathrm{C}^+$ and $\mathrm{D}$ admit the following noncommutative counterparts:

\vspace{0.1cm}

{\bf Conjecture $\mathrm{C}^+_{\mathrm{nc}}(\cA)$:} The even K\"unneth projector $\pi_+^\cA$ is {\em algebraic}, \ie there exists an endomorphism $\underline{\pi}_+^\cA$ of $U(\cA)_\bbQ$ such that $TP_\pm(\underline{\pi}_+^\cA)_{1/p}= \pi_+^\cA$.

\vspace{0.1cm}

{\bf Conjecture $\mathrm{D}_{\mathrm{nc}}(\cA)$:} The equality $K_0(\cA)_\bbQ/_{\!\sim\mathrm{hom}}=K_0(\cA)_\bbQ/_{\!\sim\mathrm{num}}$ holds.

\begin{remark}[Morita invariance]
Let $\cA$ and $\cB$ be two smooth proper dg categories. By construction, the functor $U(-)_\bbQ\colon \dgcat_{\mathrm{sp}}(k) \to \NChow(k)_\bbQ$ sends Morita equivalences to isomorphisms. Therefore, whenever $\cA$ and $\cB$ are Morita equivalent, we have $\mathrm{C}^+_{\mathrm{nc}}(\cA) \Leftrightarrow \mathrm{C}^+_{\mathrm{nc}}(\cB)$ and $\mathrm{D}^+_{\mathrm{nc}}(\cA) \Leftrightarrow \mathrm{D}^+_{\mathrm{nc}}(\cB)$.
\end{remark}

\begin{remark}[Odd K\"unneth projector]
Let $\pi_-^\cA$ be the odd K\"unneth projector of the $\bbZ/2$-graded $K$-vector space $TP_\pm(\cA)_{1/p}$. Note that if the even K\"unneth projector $\pi^\cA_+$ is algebraic, then the odd K\"unneth projector $\pi_-^\cA$ is also algebraic: simply take for $\underline{\pi}_-^\cA$ the difference $\id_{U(\cA)_\bbQ}- \underline{\pi}_+^\cA$. 
\end{remark}
\begin{remark}[Stability under tensor products]\label{rk:Kunneth}
Given smooth proper dg categories $\cA$ and $\cB$, we have the equality $\pi_+^{\cA\otimes \cB}= \pi_+^\cA \otimes \pi_+^\cB + \pi^\cA_- \otimes \pi_-^\cB$. Consequently, since $TP_\pm(-)_{1/p}$ is an additive symmetric monoidal functor, we obtain the implication: 
\begin{equation}\label{eq:implication}
\mathrm{C}^+_{\mathrm{nc}}(\cA) + \mathrm{C}^+_{\mathrm{nc}}(\cB) \Rightarrow \mathrm{C}^+_{\mathrm{nc}}(\cA\otimes \cB)
\end{equation}
Given smooth projective $k$-schemes $X$ and $Y$, the dg categories $\perf_\dg(X\times Y)$ and $\perf_\dg(X) \otimes \perf_\dg(Y)$ are Morita equivalent; see \cite[Lem.~4.26]{Gysin}. Therefore, by combining \eqref{eq:implication} with Theorem \ref{thm:main}, we obtain $\mathrm{C}^+(X) + \mathrm{C}^+(Y) \Rightarrow \mathrm{C}^+(X\times Y)$.
\end{remark}
\begin{proposition}\label{prop:fin}
Given a finite-dimensional $k$-algebra of finite global dimension $A$, the conjectures $\mathrm{C}^+_{\mathrm{nc}}(A)$ and $\mathrm{D}_{\mathrm{nc}}(A)$ hold.
\end{proposition}
\begin{proof}
The proof is similar for both cases. Hence, we will address solely conjecture $\mathrm{D}_{\mathrm{nc}}(A)$. Thanks to \cite[Thm.~3.15]{Azumaya}, we have $U(A)_\bbQ \simeq U(A/J(A))_\bbQ$, where $J(A)$ stands for the Jacobson radical of $A$. Let us write $S_1, \dots, S_m$ for the simple (right) $A/J(A)$-modules and $D_1:=\mathrm{End}_{A/J(A)}(S_1), \ldots, D_m:=\mathrm{End}_{A/J(A)}(S_m)$ for the associated division $k$-algebras. The Artin-Wedderburn theorem implies that the semi-simple quotient $A/J(A)$ is Morita equivalent to $D_1 \times \cdots \times D_m$. Moreover, the center $Z_i$ of $D_i$ is a finite field extension of $k$ and $D_i$ is a central simple $Z_i$-algebra. Making use of \cite[Thm.~2.1]{Azumaya}, we hence conclude that $U(D_i)_\bbQ \simeq U(Z_i)_\bbQ$. Consequently,  thanks to Theorem \ref{thm:main}, we obtain the equivalences of conjectures:
$$\mathrm{D}_{\mathrm{nc}}(A) \Leftrightarrow \mathrm{D}_{\mathrm{nc}}(A/J(A)) \Leftrightarrow \mathrm{D}_{\mathrm{nc}}(Z_1) + \cdots + \mathrm{D}_{\mathrm{nc}}(Z_m) \Leftrightarrow \mathrm{D}(Z_1) + \cdots + \mathrm{D}(Z_m)\,.$$
The proof follows now from the fact that $\mathrm{dim}(\mathrm{Spec}(Z_i))=0$ for every $i$.
\end{proof}
%-------------------------------------------------------------------------------
\section{Proof of Theorem \ref{thm:main}}
%-------------------------------------------------------------------------------
%-------------------------------------------------------------------------------
\subsection*{Type $\mathrm{C}^+$}
%-------------------------------------------------------------------------------
Let $\Chow(k)_\bbQ$ be the classical category of Chow motives; see Manin \cite{Manin}. By construction, this $\bbQ$-linear category is additive, rigid symmetric monoidal, and comes equipped with a symmetric monoidal functor $\mathfrak{h}(-)_\bbQ\colon \mathrm{SmProj}(k)^\op \to \Chow(k)_\bbQ$ defined on smooth projective $k$-schemes. Crystalline cohomology gives rise to a symmetric monoidal functor $H^\ast_{\mathrm{crys}}\colon \Chow(k)_\bbQ \to \mathrm{cect}_\bbZ(K)$ with values in the category of finite-dimensional $\bbZ$-graded $K$-vector spaces. By composing it with the functor $\mathrm{vect}_\bbZ(K) \to \mathrm{vect}_{\bbZ/2}(K)$ that sends $\{V_i\}_{i \in \bbZ}$ to $(\bigoplus_{i\,\text{even}} V_i, \bigoplus_{i \, \text{odd}} V_i)$, we hence obtain the following $\bbQ$-linear symmetric monoidal functor:
\begin{equation}\label{eq:2-periodic}
\Chow(k)_\bbQ \too \mathrm{vect}_{\bbZ/2}(K)  \quad \quad  \mathfrak{h}(X)_\bbQ \mapsto (\bigoplus_{i\,\text{even}}H^i_{\mathrm{crys}}(X), \bigoplus_{i\,\text{odd}}H^i_{\mathrm{crys}}(X))\,.
\end{equation}
Recall from \cite[Thm.~4.3]{book} that there exists a $\bbQ$-linear, fully-faithful, symmetric monoidal functor $\Phi$ making the following diagram commute
\begin{equation}\label{eq:diagram-big}
\xymatrix{
\mathrm{SmProj}(k)^\op \ar[rr]^-{X\mapsto \perf_\dg(X)} \ar[d]_-{\mathfrak{h}(-)_\bbQ} && \dgcat_{\mathrm{sp}}(k) \ar[dd]^-{U(-)_\bbQ} \\
\Chow(k)_\bbQ \ar[d]_-\iota && \\
\Chow(k)_\bbQ/_{\!-\otimes \bbQ(1)} \ar[rr]_-{\Phi} && \NChow(k)_\bbQ\,,
}
\end{equation}
where $\Chow(k)_\bbQ/_{\!-\otimes \bbQ(1)}$ stands for the orbit category with respect to the Tate motive $\bbQ(1)$; see \S\ref{sub:orbit}. Consider the following composition:
 \begin{equation}\label{eq:composition}
 \Chow(k)_\bbQ \stackrel{\iota}{\too} \Chow(k)_\bbQ/_{\!-\otimes \bbQ(1)} \stackrel{\Phi}{\too} \NChow(k)_\bbQ \stackrel{TP_\pm(-)_{1/p}}{\too} \mathrm{vect}_{\bbZ/2}(K)\,.
 \end{equation}
%By assumption, $X$ is a smooth projective $k$-scheme. Therefore, thanks to Theorem \ref{thm:TP}, we have a natural isomorphism:
%\begin{equation}\label{eq:HKR}
%TP_\pm(\perf_\dg(X))_{1/p} \simeq (\bigoplus_{i\,\text{even}}H^i_{\mathrm{crys}}(X), \bigoplus_{i\,\text{odd}}H^i_{\mathrm{crys}}(X))\,.
%\end{equation}
We now have all the ingredients necessary to prove the equivalence of conjectures $\mathrm{C}^+(X) \Leftrightarrow \mathrm{C}^+_{\mathrm{nc}}(\perf_\dg(X))$. Assume first that the conjecture $\mathrm{C}^+(X)$ holds, \ie that there exists an endomorphism $\underline{\pi}^+_X$ of the Chow motive $\mathfrak{h}(X)_\bbQ$ such that $H^\ast_{\mathrm{crys}}(\underline{\pi}_X^+)= \pi^+_X$. Thanks to the natural isomorphism \eqref{eq:Scholze} and to the commutative diagram \eqref{eq:diagram-big}, the composition \eqref{eq:composition} is naturally isomorphic to the above functor \eqref{eq:2-periodic}. Consequently, by taking the image of $\underline{\pi}^+_X$ under the composition $\Phi \circ \iota$, we conclude that the conjecture $\mathrm{C}^+_{\mathrm{nc}}(\perf_\dg(X))$ also holds. 

Assume now that the conjecture $\mathrm{C}^+_{\mathrm{nc}}(\perf_\dg(X))$ holds, \ie that there exists an endomorphism $\underline{\pi}_+$ of $U(\perf_\dg(X))_\bbQ$ such that $TP_\pm(\underline{\pi}_+)_{1/p}= \pi_+^{\perf_\dg(X)}$. Thanks to the commutativity of the diagram \eqref{eq:diagram-big} and to the fully-faithfulness of the functor $\Phi$, the endomorphism $\underline{\pi}_+$ corresponds to an endomorphism $\{\underline{\pi}_n^+\}_{n \in \bbZ}$ of the object $\iota(\mathfrak{h}(X)_\bbQ)$ in the orbit category $\Chow(k)_\bbQ/_{\!-\otimes \bbQ(1)}$. Moreover, since the composition \eqref{eq:composition} is naturally isomorphic to \eqref{eq:2-periodic}, the image of $\{\underline{\pi}_n^+\}_{n \in \bbZ}$ under the composition $TP_\pm(-)_{1/p} \circ \Phi$ agrees with the endomorphism $(\id,0)$ of the $\bbZ/2$-graded $K$-vector space $(\bigoplus_{i\,\mathrm{even}} H^i_{\text{crys}}(X), \bigoplus_{i\,\mathrm{odd}} H^i_{\text{crys}}(X))$. Note that the following morphism 
$$\underline{\pi}_n^+\colon \mathfrak{h}(X)_\bbQ \too \mathfrak{h}(X)_\bbQ(n)$$
in $\Chow(k)_\bbQ$, where $\mathfrak{h}(X)_\bbQ(n)$ stands for $\mathfrak{h}(X)_\bbQ \otimes \bbQ(1)^{\otimes n}$, induces an homomorphism of degree $-2n$ in crystalline cohomology theory:
$$H^\ast_{\text{crys}}(\underline{\pi}_n^+)\colon H^\ast_{\text{crys}}(X) \too H^{\ast - 2n}_{\text{crys}}(X)\,.$$ 
Since the image of $\{\underline{\pi}_n^+\}_{n \in \bbZ}$ under the composition $TP_\pm(-)_{1/p} \circ \Phi$ is given by $\sum_n H^\ast_{\text{crys}}(\underline{\pi}_n^+)$, this implies that all the homomorphisms $H^\ast_{\text{crys}}(\underline{\pi}_n^+)$, with $n \neq 0$, are necessarily equal to zero. Consequently, $\underline{\pi}_0^+$ is an endomorphism of the Chow motive $\mathfrak{h}(X)_\bbQ$ whose image under the functor \eqref{eq:2-periodic} agrees with the above endomorphism $(\id,0)$ of the $\bbZ/2$-graded $K$-vector space $(\bigoplus_{i\,\mathrm{even}} H^i_{\text{crys}}(X), \bigoplus_{i\,\mathrm{odd}} H^i_{\text{crys}}(X))$. By construction of the functor \eqref{eq:2-periodic}, we hence conclude finally that the image of $\underline{\pi}_0^+$ under the functor $H^\ast_{\text{crys}}\colon \Chow(k)_\bbQ \to \mathrm{vect}(K)$ agrees with the even K\"unneth projector $\pi^+_X:=\sum_i \pi_X^{2i}$. This proves the conjecture $\mathrm{C}^+(X)$.
%-------------------------------------------------------------------------------
\subsection*{Type $\mathrm{D}$}
%-------------------------------------------------------------------------------
Let $\Hom(k)_\bbQ$ be the classical category of homological motives (with respect to crystalline cohomology theory) and $\Num(k)_\bbQ$ the classical category of numerical motives. Recall from \cite[\S4.6]{book} that there exists a $\bbQ$-linear, fully-faithful, symmetric monoidal functor $\Phi_N$ making the following diagram commute: 
\begin{equation}\label{eq:comp1}
\xymatrix{
\Chow(k)_\bbQ \ar[d] \ar[r]^-\iota & \Chow(k)_\bbQ/_{\!-\otimes \bbQ(1)} \ar[d] \ar[r]^-\Phi & \NChow(k)_\bbQ \ar[d] \\
\Num(k)_\bbQ \ar[r]^-\iota & \Num(k)_\bbQ/_{\!-\otimes \bbQ(1)} \ar[r]^-{\Phi_N}& \NNum(k)_\bbQ\,.
}
\end{equation}
By construction, the kernel of crystalline cohomology $H^\ast_{\text{crys}}\colon \Chow(k)_\bbQ \to \mathrm{vect}(K)$ agrees with the kernel of the above symmetric monoidal functor \eqref{eq:2-periodic}. Therefore, since the composition \eqref{eq:2-periodic} is naturally isomorphic to \eqref{eq:composition} and \eqref{eq:N} is the largest $\otimes$-ideal of the categories $\Chow(k)_\bbQ$ and $\NChow(k)_\bbQ$, the preceding diagram \eqref{eq:comp1} admits the following ``factorization''
\begin{equation}\label{eq:comp2}
\xymatrix{
\Chow(k)_\bbQ \ar[d] \ar[r]^-\iota & \Chow(k)_\bbQ/_{\!-\otimes \bbQ(1)} \ar[d] \ar[r]^-\Phi & \NChow(k)_\bbQ \ar[d] \\
\Hom(k)_\bbQ \ar[d] \ar[r]^-\iota & \Hom(k)_\bbQ/_{\!-\otimes \bbQ(1)} \ar[d] \ar[r]^-{\Phi_H} & \NHom(k)_\bbQ \ar[d] \\
\Num(k)_\bbQ \ar[r]^-\iota & \Num(k)_\bbQ/_{\!-\otimes \bbQ(1)} \ar[r]^-{\Phi_N} & \NNum(k)_\bbQ\,,
}
\end{equation}
where $\Phi_H$ stands for the functor induced by the universal property of the orbit category $\Hom(k)_\bbQ/_{\!-\otimes \bbQ(1)}$.
\begin{lemma}\label{lem:aux}
The induced functor $\Phi_H$ is full.
\end{lemma}
\begin{proof}
Let us write $(\Chow(k)_\bbQ/_{\!-\otimes \bbQ(1)})/\mathrm{Ker}$ for the idempotent completion of the quotient of the orbit category $\Chow(k)_\bbQ/_{\!-\otimes \bbQ(1)}$ by the kernel of the composition $TP_\pm(-)_{1/p} \circ \Phi$. Under this notation, we have the following commutative diagram
$$
\xymatrix{
\Chow(k)_\bbQ/_{\!-\otimes \bbQ(1)} \ar[d] \ar@{=}[r] & \Chow(k)_\bbQ/_{\!-\otimes \bbQ(1)} \ar[d] \ar[r]^-\Phi & \NChow(k)_\bbQ \ar[d] \\
\Hom(k)_\bbQ/_{\!-\otimes \bbQ(1)} \ar[r]^-\theta & (\Chow(k)_\bbQ/_{\!-\otimes \bbQ(1)})/\mathrm{Ker} \ar[r]^-{\Phi'_H}& \NHom(k)_\bbQ\,,
}
$$
where $\theta$, resp. $\Phi'_H$, stands for the canonical, resp. induced, functor. The proof follows now from the fact that the functor $\theta$, resp. $\Phi_H$, is full (see \cite[Lem.~4.7]{JEMS}), resp. fully-faithful, and that $\Phi_H=\Phi'_{H} \circ \theta$.
\end{proof}
Thanks to the commutative diagram \eqref{eq:diagram-big}, the bottom right-hand side square in \eqref{eq:comp2} yields  the following commutative square of $\bbQ$-vector spaces:
$$
\xymatrix{
\Hom_{\Hom(k)_\bbQ/_{\!-\otimes \bbQ(1)}}(\mathfrak{h}(k)_\bbQ, \mathfrak{h}(X)_\bbQ) \ar@{->>}[d] \ar@{->>}[r] & \Hom_{\NHom(k)_\bbQ}(U(k)_\bbQ, U(\perf_\dg(X))_\bbQ) \ar@{->>}[d]  \\
\Hom_{\Num(k)_\bbQ/_{\!-\otimes \bbQ(1)}}(\mathfrak{h}(k)_\bbQ, \mathfrak{h}(X)_\bbQ) \ar[r]^-\simeq & \Hom_{\NNum(k)_\bbQ}(U(k)_\bbQ, U(\perf_\dg(X))_\bbQ)\,.
}
$$
Note that thanks to Lemma \ref{lem:aux}, the upper horizontal homomorphism is surjective. Note also that by construction of the categories of (noncommutative) homological and numerical motives, the preceding commutative square identifies with 
\begin{equation}\label{eq:square1}
\xymatrix{
Z^\ast(X)_\bbQ/_{\!\sim \mathrm{hom}} \ar@{->>}[r] \ar@{->>}[d] & K_0(\perf_\dg(X))_\bbQ/_{\!\sim \mathrm{hom}} \ar@{->>}[d] \\
Z^\ast(X)_\bbQ/_{\!\sim \mathrm{num}} \ar[r]^-\simeq & K_0(\perf_\dg(X))_\bbQ/_{\!\sim \mathrm{num}}\,.
}
\end{equation}
We now have all the ingredients necessary to prove the equivalence of conjectures $\mathrm{D}(X) \Leftrightarrow \mathrm{D}_{\mathrm{nc}}(\perf_\dg(X))$. Assume first that the conjecture $\mathrm{D}(X)$ holds, \ie that the vertical left-hand side homomorphism in \eqref{eq:square1} is injective. From the commutativity of \eqref{eq:square1}, we conclude that the vertical right-hand side homomorphism in \eqref{eq:square1} is also injective, \ie that the conjecture $\mathrm{D}_{\mathrm{nc}}(\perf_\dg(X))$ also holds. 

Assume now that the conjecture $\mathrm{D}_{\mathrm{nc}}(\perf_\dg(X))$ holds, \ie that vertical right-hand side homomorphism in \eqref{eq:square1} is injective. Note that the vertical left-hand side homomorphism in \eqref{eq:square1} is the diagonal quotient homomorphism from the direct sum 
$\bigoplus_{i=0}^{\mathrm{dim}(X)} Z^i(X)_\bbQ/_{\!\sim \mathrm{hom}}$ to the direct sum $\bigoplus_{i=0}^{\mathrm{dim}(X)} Z^i(X)_\bbQ/_{\!\sim \mathrm{num}}$. Therefore, thanks to the commutativity of \eqref{eq:square1}, in order to prove the conjecture $\mathrm{D}(X)$ it suffices to show that the following homomorphisms are injective:
\begin{eqnarray}\label{eq:homomorphisms}
Z^i(X)_\bbQ/_{\!\sim \mathrm{hom}} \too K_0(\perf_\dg(X))_\bbQ/_{\!\sim \mathrm{hom}} && 0 \leq i \leq \mathrm{dim}(X)\,.
\end{eqnarray}
Note that the composed functor $\Phi_H \circ \iota$ in \eqref{eq:square1} is faithful. In particular, for every $0 \leq i \leq \mathrm{dim}(X)$, the induced homomorphism is injective:
\begin{equation}\label{eq:final}
\Hom_{\Hom(k)_\bbQ}(\mathfrak{h}(k)_\bbQ, \mathfrak{h}(X)_\bbQ(i)) \too \Hom_{\NHom(k)_\bbQ}(U(k)_\bbQ, U(\perf_\dg(X))_\bbQ)\,.
\end{equation}
By construction of the category of (noncommutative) homological motives and of the orbit category $\Hom(k)_\bbQ/_{\!-\otimes \bbQ(1)}$, the preceding homomorphisms \eqref{eq:final} (induced by the functor $\Phi_H \circ \iota$) correspond to the above homomorphisms \eqref{eq:homomorphisms} (induced by the functor $\Phi_H$). This implies that the homomorphisms \eqref{eq:homomorphisms} are injective, and hence proves the conjecture $\mathrm{D}(X)$.
%-------------------------------------------------------------------------------
\section{Proof of Theorem \ref{thm:HPD}}
%-------------------------------------------------------------------------------
Thanks to Theorem \ref{thm:main}, the proof of Theorem \ref{thm:HPD} is similar to the proof of \cite[Thm.~1.4]{CD}. Simply replace $\cO_X(r)$ by $\cL_X(r)$, $\perf(Y;\cF)$ by $\perf(Y)$, and the references \cite[Thm.~6.3]{Kuznetsov-IHES} and \cite[\S2.4]{Kuznetsov-ICM} (in characteristic zero) by the reference \cite[Thm.~2.3.4]{Bernardara} (in arbitrary characteristic).
%%-------------------------------------------------------------------------------
\section{Proof of Theorem \ref{thm:orbifold}}
%%-------------------------------------------------------------------------------
We start by proving the first claim. Since by assumption $p\nmid n$, \ie since $1/n \in k$, it follows from \cite[Thm.~1.1 and Rk.~1.4]{Orbifold} that the noncommutative Chow motive $U(\perf_\dg(\cX))_\bbQ$ is a direct summand of $\bigoplus_{\sigma \subseteq G} U(\perf_\dg(X^\sigma \times \mathrm{Spec}(k[\sigma])))_\bbQ$. By definition, the noncommutative standard conjectures of type $\mathrm{C}^+$ and $\mathrm{D}$ are stable under direct sums and direct summands. Therefore, we obtain the implications:
$$ \sum_{\sigma \subseteq G} \mathrm{C}^+_{\mathrm{nc}}(\perf_\dg(X^\sigma \times \mathrm{Spec}(k[\sigma]))) \Rightarrow \mathrm{C}^+_{\mathrm{nc}}(\perf_\dg(\cX))=: \mathrm{C}^+(\cX)$$  
$$ \sum_{\sigma \subseteq G} \mathrm{D}_{\mathrm{nc}}(\perf_\dg(X^\sigma \times \mathrm{Spec}(k[\sigma]))) \Rightarrow \mathrm{D}_{\mathrm{nc}}(\perf_\dg(\cX))=:\mathrm{D}(\cX)\,.$$
The proof is now a consequence of Theorem \ref{thm:main} and of the following implications
$$ \mathrm{C}^+(X^\sigma) \stackrel{\text{(a)}}{\Rightarrow} \mathrm{C}^+(X^\sigma) + \mathrm{C}^+(\mathrm{Spec}(k[\sigma])) \stackrel{\text{(b)}}{\Rightarrow} \mathrm{C}^+(X^\sigma \times \mathrm{Spec}(k[\sigma]))\,,$$
where (a) follows from the fact that $\mathrm{dim}(\mathrm{Spec}(k[\sigma]))=0$ and (b) from Remark~\ref{rk:Kunneth}. 
  
Let us now prove the second claim. If $k$ contains moreover the $n^{\mathrm{th}}$ roots of unity of $k$, then it follows from \cite[Cor.~1.6(i)]{Orbifold} that the noncommutative Chow motive $U(\perf_\dg(\cX))_\bbQ$ is a direct summand of $\bigoplus_{\sigma \subseteq G} U(\perf_\dg(X^\sigma))_\bbQ$. Therefore, since the noncommutative standard conjecture of type $\mathrm{D}$ is stable under direct sums and direct summands, the proof is a consequence of the following implications
$$ \sum_{\sigma \subseteq G} \mathrm{D}(X^\sigma) \stackrel{\text{(a)}}{\Rightarrow} \sum_{\sigma \subseteq G} \mathrm{D}_{\mathrm{nc}}(\perf_\dg(X^\sigma)) \Rightarrow \mathrm{D}_{\mathrm{nc}}(\perf_\dg(\cX))=:\mathrm{D}(\cX)\,,$$
where (a) follows from Theorem \ref{thm:main}. 
%-------------------------------------------------------------------------------
\section{Application III: Zeta functions of endomorphisms}\label{sec:zeta}
%-------------------------------------------------------------------------------
Let $\cA$ be a smooth proper dg category and $f$ an endomorphism of the noncommutative Chow motive $U(\cA)_\bbQ$; see \S\ref{sec:NC}. Following \cite[\S5]{Positive}, the {\em zeta function of $f$} is defined as the following formal power series 
$$Z(f;t):=\mathrm{exp}\left(\sum_{n \geq 1} \mathrm{tr}(f^{\circ n}) \frac{t^n}{n}\right) \in \bbQ\llbracket t \rrbracket\,,$$ 
where $f^{\circ n}$ stands for the composition of $f$ with itself $n$-times, $\mathrm{tr}(f^{\circ n})$ stands for the categorical trace of $f^{\circ n}$, and $\mathrm{exp}(t):=\sum_{m \geq 0} \frac{t^m}{m!} \in \bbQ\llbracket t\rrbracket$. Recall from \cite[Rk.~5.2]{Positive} that when $f=[\mathrm{B}]_\bbQ$ with $\mathrm{B} \in \cD_c(\cA^\op \otimes \cA)$, we have the computation
\begin{equation}\label{eq:integers}
\mathrm{tr}(f^{\circ n}) =[HH(\cA; \underbrace{\mathrm{B}\otimes^{\bf L}_\cA \cdots \otimes^{\bf L}_\cA \mathrm{B}}_{n\text{-}\text{times}})] \in K_0(k) \simeq \bbZ\,,
\end{equation}
where $HH(\cA; \mathrm{B}\otimes^{\bf L}_\cA \cdots \otimes^{\bf L}_\cA \mathrm{B})$ stands for the Hochschild homology of $\cA$ with coefficients in the dg $\cA\text{-}\cA$ bimodule $\mathrm{B}\otimes^{\bf L}_\cA \cdots \otimes^{\bf L}_\cA \mathrm{B}$.
\begin{example}[Classical zeta function]\label{ex:Frobenius}
Let $k=\bbF_q$ be a finite field of characteristic $p$, $X$ a smooth projective $k$-scheme, and $\mathrm{Fr}$ the Frobenius of $X$. Recall from \cite[Example~5.4]{Positive} that when $\cA= \perf_\dg(X)$ and $\mathrm{B}$ is the dg bimodule associated to the pull-back dg functor $\mathrm{Fr}^\ast\colon \perf_\dg(X) \to \perf_\dg(X)$, the above integer \eqref{eq:integers} agrees with $|X(\bbF_{q^n})|$. Consequently, in this particular case, the zeta function of $f=[\mathrm{B}]_\bbQ$ reduces to the classical zeta function $Z_X(t):= \mathrm{exp}(\sum_{n \geq 1} |X(\bbF_{q^n})| \frac{t^n}{n})$ of $X$.
\end{example}
As proved in \cite[Thm.~5.8]{Positive}, the formal power series $Z(f;t)$ is rational and satisfies a functional equation. In the particular case of Example \ref{ex:Frobenius}, these results yield an alternative proof of ``half'' of the Weil conjectures; consult \cite[Cor.~5.12]{Positive} for details. In {\em loc. cit.}, we established moreover the following equality:
\begin{equation}\label{eq:zeta2}
 Z(f;t) = \frac{\mathrm{det}(\id - t \,TP_-(f)_{1/p}\,|\, TP_-(\cA)_{1/p})}{\mathrm{det}(\id - t\, TP_+(f)_{1/p}\,|\, TP_+(\cA)_{1/p})} \in K(t)\,.
\end{equation} 
%where $TP_\pm(-)_{1/p}$ stands for topological periodic cyclic homology (see \S\ref{sec:NC-conj}).
\begin{theorem}\label{thm:zeta}
If the conjecture $\mathrm{C}^+_{\mathrm{nc}}(\cA)$ holds, then the numerator and denominator of \eqref{eq:zeta2} are polynomials with $\bbQ$-coefficients. Moreover, when $f=[\mathrm{B}]_\bbQ$ with $\mathrm{B} \in \cD_c(\cA^\op \otimes \cA)$, the same holds with $\bbZ$-coefficients.
\end{theorem}
Note that thanks to Corollary \ref{cor:examples}, resp. Proposition \ref{prop:fin}, the preceding Theorem \ref{thm:zeta} can be applied, for example, to any ``low-dimensional'' orbifold, resp. to any finite dimensional $k$-algebra of finite global dimension.
\begin{corollary}\label{cor:zeta}
Let $X$ be a smooth projective $\bbF_q$-scheme.
\begin{itemize}
\item[(i)] We have the following equality:
\begin{equation}\label{eq:zeta3}
 Z_X(t) = \frac{\prod_{i \,\mathrm{odd}}\mathrm{det}(\id - t \,H^i_{\mathrm{crys}}(\mathrm{Fr})\,|\, H^i_{\mathrm{crys}}(X))}{\prod_{i \,\mathrm{even}}\mathrm{det}(\id - t \,H^i_{\mathrm{crys}}(\mathrm{Fr})\,|\, H^i_{\mathrm{crys}}(X))} \in K(t)\,.
\end{equation} 
\item[(ii)] If the conjecture $\mathrm{C}^+(X)$ holds, then the numerator and denominator of \eqref{eq:zeta3} are polynomials with $\bbZ$-coefficients.
\end{itemize}
\end{corollary}
\begin{proof}
On the one hand, item (i) follows from the combination of \eqref{eq:zeta2} with Example \ref{ex:Frobenius} and Theorem \ref{thm:TP}. On the other hand, item (ii) follows from the combination of Theorems \ref{thm:main}, \ref{thm:TP}, and \ref{thm:zeta}, with Example \ref{ex:Frobenius}.
\end{proof}
On the one hand, item (i) is Berthelot's cohomological interpretation of the classical zeta function in terms of crystalline cohomology theory; see \cite[page 583]{Berthelot}. On the other hand, item (ii) is Grothendieck's conditional approach to ``half''\footnote{The other ``half'' of the Riemann hypothesis asserts that the roots of the characteristic polynomial $\mathrm{det}(\id - t \,H^i_{\mathrm{crys}}(\mathrm{Fr})\,|\, H^i_{\mathrm{crys}}(X))$ have absolute value $q^{i/2}$.} of the Riemann hypothesis; see \cite[\S1-2]{Grothendieck} and \cite[4.1 Theorem]{Kleim}. Corollary \ref{cor:zeta} provides us with an alternative proof of these important results. Moreover, the above equality \eqref{eq:zeta2}, resp. Theorem \ref{thm:zeta}, establishes a far-reaching noncommutative generalization of Berthelot's cohomological interpretation of the classical zeta function, resp. of Grothendieck's conditional approach to ``half'' of the Riemann hypothesis.

\subsection*{Proof of Theorem \ref{thm:zeta}}
If the conjecture $\mathrm{C}^+_{\mathrm{nc}}(\cA)$ holds, then there exists an endomorphism $\underline{\pi}^\cA_+$ of $U(\cA)_\bbQ$ such that $TP_\pm(\underline{\pi}^\cA_+)=\pi^\cA_+$. In what follows, we write $f_+$ for the composition $\underline{\pi}^\cA_+\circ f$. Note that $TP_\pm(f_+)_{1/p}=TP_+(f)_{1/p}$.

We start by proving the first claim. Thanks to the classical Newton identities, the coefficients of the characteristic polynomial $\mathrm{det}(\id - t\, TP_+(f)_{1/p}\,|\, TP_+(\cA)_{1/p})$ can be written as polynomials with $\bbQ$-coefficients in the power sums $\alpha_1^n + \cdots + \alpha_r^n$, $n \geq 1$, where $\alpha_1, \dots, \alpha_r$ are the eigenvalues (with multiplicities) of the $K$-linear homomorphism $TP_+(f)_{1/p}$. Therefore, it suffices to show that these power sums are rational numbers. This follows from the following equalities 
\begin{equation}\label{eq:equalities}
\alpha_1^n + \cdots + \alpha_r^n = \mathrm{tr}(TP_+(f^{\circ n})_{1/p}) = \mathrm{tr}(TP_+(f_+^{\circ n})_{1/p})= \mathrm{tr}(f^{\circ n}_+)
\end{equation}
and from the fact that $\mathrm{tr}(f^{\circ n}_+) \in \bbQ$.

Let us now prove the second claim. By construction, the category of noncommutative Chow motives with $\bbZ$-coefficients $\NChow(k)_\bbZ$ (see \cite[Rk.~4.2]{book}) is additive, rigid symmetric monoidal, and comes equipped with a symmetric monoidal functor $U(-)_\bbZ\colon \dgcat_{\mathrm{sp}}(k) \to \NChow(k)_\bbZ$ as well as with a $\bbQ$-linearization symmetric monoidal functor $(-)_\bbQ\colon \NChow(k)_\bbZ \to \NChow(k)_\bbQ$. Therefore, if $f=[\mathrm{B}]_\bbQ$ with $\mathrm{B} \in \cD_c(\cA^\op \otimes \cA)$, \ie if $f$ is the $\bbQ$-linearization of an endomorphism $[\mathrm{B}]$ of $U(\cA)_\bbZ$, then $\mathrm{tr}(f^{\circ n})=\mathrm{tr}([\mathrm{B}]^{\circ n})\in \bbZ$ for every $n \geq 1$. The endomorphism $\underline{\pi}^\cA_+$ of $U(\cA)_\bbQ$ is not necessarily the $\bbQ$-linearization of an endomorphism of $U(\cA)_\bbZ$. Nevertheless, by removing denominators, there exists an integer $\lambda>0$ such that $\lambda\cdot \underline{\pi}^\cA_+=[\mathrm{B}']_\bbQ$ for some dg $\cA\text{-}\cA$ bimodule $\mathrm{B}'\in \cD_c(\cA^\op \otimes \cA)$. Consequently, making use of the following equalities
$$ \lambda \cdot \mathrm{tr}(f^{\circ n}_+)= \lambda \cdot \mathrm{tr}(\underline{\pi}^\cA_+\circ f^{\circ n})= \mathrm{tr}((\lambda \cdot \underline{\pi}^\cA_+)\circ f^{\circ n}) = \mathrm{tr}([\mathrm{B}']\circ [\mathrm{B}]^{\circ n})\,,$$
we conclude that $\lambda\cdot \mathrm{tr}(f^{\circ n}_+)\in \bbZ$ for every $n \geq 1$. Thanks to the above equalities \eqref{eq:equalities}, \cite[2.8 Lemma]{Kleim} hence implies that the coefficients of the characteristic polynomial $\mathrm{det}(\id - t\, TP_+(f)_{1/p}\,|\, TP_+(\cA)_{1/p})$ are algebraic integers. Since these numbers are also rational, we conclude that they are necessarily integers.

Finally, note that the proof concerning the coefficients of the characteristic polynomial $\mathrm{det}(\id - t\, TP_-(f)_{1/p}\,|\, TP_-(\cA)_{1/p})$ is similar: simply replace $\underline{\pi}^\cA_+$ by $\underline{\pi}^\cA_-$.
%%-------------------------------------------------------------------------------
%\subsection*{Applications to the classical zeta function}
%%-------------------------------------------------------------------------------
%By combining Theorems \ref{thm:main}, \ref{thm:TP}, and \ref{thm:zeta}, with the above Example \ref{ex:Frobenius}, we obtain the following result:
%\begin{corollary}\label{cor:zeta}
%Let $X$ be a smooth projective $\bbF_q$-scheme.
%\begin{itemize}
%\item[(i)] We have the following equality:
%\begin{equation}\label{eq:zeta3}
% Z_X(t) = \frac{\prod_{i \,\mathrm{odd}}\mathrm{det}(\id - t \,H^i_{\mathrm{crys}}(\mathrm{Fr})\,|\, H^i_{\mathrm{crys}}(X))}{\prod_{i \,\mathrm{even}}\mathrm{det}(\id - t \,H^i_{\mathrm{crys}}(\mathrm{Fr})\,|\, H^i_{\mathrm{crys}}(X))} \in K(t)\,.
%\end{equation} 
%\item[(ii)] If the conjecture $\mathrm{C}^+(X)$ holds, then the numerator and denominator of \eqref{eq:zeta3} are polynomials with $\bbZ$-coefficients.
%\end{itemize}
%\end{corollary}
%The equality \eqref{eq:zeta3} was originally established by Berthelot in \cite[page 583]{Berthelot}. Corollary \ref{cor:zeta} provides us with an alternative proof of this cohomological interpretation of the classical zeta function in terms of crystalline cohomology theory.

\medbreak\noindent\textbf{Acknowledgments:} After the release of \cite{CD}, Bruno Kahn (motivated by the fact that the standard conjecture of Hodge type is wide open in positive characteristic; see Remark \ref{rk:zero}) asked me if similar results would hold in positive characteristic. In this article, making use of topological periodic cyclic homology, I answer affirmatively to Kahn's question. I thank him for this motivating question. I am also very grateful to Lars Hesselholt for useful discussions concerning topological periodic cyclic homology and to Peter Scholze for explaining me the proof of Theorem \ref{thm:TP}.


\begin{thebibliography}{00}


\bibitem{AMN} B.~Antieau, A.~Mathew and T.~Nikolaus, {\em On the Blumberg-Mandell K\"unneth theorem for $TP$}. Available at arXiv:1710.05658.

\bibitem{Bernardara} A.~Auel, M.~Bernardara and M.~Bolognesi, {\em Fibrations in complete intersections of quadrics, Clifford algebras, derived categories, and rationality problems}. J.~Math.~Pures Appl. (9)~{\bf 102} (2014), no.~1, 249--291.

\bibitem{BBF} M.~Bernardara, M.~Bolognesi and D.~Faenzi, {\em Homological projective duality for determinantal varieties}. Adv. Math. {\bf 296} (2016), 181--209.

\bibitem{Berthelot} P.~Berthelot, {\em Cohomologie cristalline des sch\'emas de caract\'eristique $p>0$}. Lecture Notes in Math. {\bf 407}, Springer-Verlag, New York, 1974.

\bibitem{BMS2} B.~Bhatt, M.~ Morrow and P.~Scholze, {\em Topological Hochschild homology and integral $p$-adic Hodge theory}. Available at arXiv:1802.03261.

\bibitem{BM} A.~Blumberg and M.~Mandell, {\em The strong K\"unneth theorem for topological periodic cyclic homology}. Available at arXiv:1706.06846.

\bibitem{BLV} R.~Buchweitz, G.~Leuschke and M.~Van den Bergh, {\em On  the derived category of Grassmannians in arbitrary characteristic}. Comp. Math. {\bf 151} (2015), no.~7, 1242--1264.

\bibitem{Faenzi} A.~Conca and D.~Faenzi {\em A remark on hyperplane sections of rational normal scrolls}. Available at arXiv:1709.08332.

\bibitem{Grothendieck} A.~Grothendieck, {\em Standard conjectures on algebraic cycles}. 1969 Algebraic Geometry (Internat. Colloq., Tata Inst. Fund. Res., Bombay, 1968) pp. 193--199 Oxford Univ. Press, London.

\bibitem{Harris} J.~Harris, {\em Algebraic geometry.  A first course}. Graduate Texts in Mathematics, {\bf 133}. Springer-Verlag, New York, 1992.

\bibitem{Hesselholt} L.~Hesselholt, {\em Topological periodic cyclic homology and the Hasse-Weil zeta function}. Available at arXiv:1602.01980.

\bibitem{KM} N.~Katz and W.~Messing, {\em Some consequences of the Riemann hypothesis for varieties over finite fields}. 
Invent. Math. {\bf 23} (1974), 73--77. 

\bibitem{ICM-Keller} B.~Keller, {\em On differential graded categories}. International Congress of Mathematicians (Madrid), Vol.~II,  151--190. Eur.~Math.~Soc., Z{\"u}rich (2006).

\bibitem{Kleim1} S.~L.~Kleiman, {\em The standard conjectures}. Motives (Seattle, WA, 1991), 3--20, 
Proc. Sympos. Pure Math., {\bf 55}, Part 1, Amer. Math. Soc., Providence, RI, 1994. 

\bibitem{Kleim} \bysame, {\em Algebraic cycles and the Weil conjectures}. Dix expos\'es sur la cohomologie des sch\'emas, 359--386, Adv. Stud. Pure Math., {\bf 3}, North-Holland, Amsterdam, 1968. 

\bibitem{Miami} M.~Kontsevich, {\em Mixed noncommutative motives}. Talk at the Workshop on Homological Mirror Symmetry,  Miami, 2010. Notes available at \url{www-math.mit.edu/auroux/frg/miami10-notes}.  
%%
\bibitem{finMot} \bysame, {\em Notes on motives in finite characteristic}.  Algebra, arithmetic, and geometry: in honor of Yu. I. Manin. Vol. II,  213--247, Progr. Math., {\bf 270}, BirkhŠuser Boston, MA, 2009. 
%%
\bibitem{IAS} \bysame, {\em Noncommutative motives}. Talk at the IAS on the occasion of the $61^{\mathrm{st}}$ birthday of Pierre Deligne (2005). Available at \url{http://video.ias.edu/Geometry-and-Arithmetic}.    

\bibitem{Kuznetsov-ICM} A.~Kuznetsov, {\em Semiorthogonal decompositions in algebraic geometry}. Available at 1404.3143. To appear in Proceedings of the ICM 2014.

\bibitem{Kuznetsov-IHES} \bysame, {\em Homological projective duality}. Publ. Math. IH\'ES (2007), no. {\bf 105}, 157--220.

\bibitem{LO} V.~Lunts and D.~Orlov, {\em Uniqueness of enhancement for triangulated categories}. J. Amer. Math. Soc. {\bf 23} (2010), no.~3, 853--908.

\bibitem{Manin} Y.~I.~Manin, {\em Correspondences, motifs and monoidal transformations}. Mat. Sb. (N.S.) {\bf 77} (119) (1968), 475--507.

\bibitem{JEMS} M.~Marcolli and G.~Tabuada, {\em Noncommutative numerical motives, Tannakian structures, and motivic Galois groups}. Journal of the EMS {\bf 18} (2016), 623--655.

\bibitem{Scholze} P.~Scholze, Private communication at the Arbeitsgemeinschaft {\em Topological Cyclic Homology} in Oberwolfach, April 2018.

\bibitem{CD} G.~Tabuada, {\em A note on Grothendieck's standard conjectures of type $\mathrm{C}^+$ and $\mathrm{D}$}. Proceedings of the American Mathematical Society {\bf 146} (2018), no. 4, 1389--1399.

\bibitem{book} \bysame, {\em Noncommutative Motives}. With a preface by Yuri I. Manin. University Lecture
Series, {\bf 63}. American Mathematical Society, Providence, RI, 2015.

\bibitem{Tate} \bysame, {\em $HPD$-invariance of the Tate conjecture, Beilinson and Parshin conjectures}. Available at arXiv:1707.06639.

\bibitem{Positive} \bysame, {\em Noncommutative motives in positive characteristic and their applications}. Available at arXiv:1707.04248.

\bibitem{survey} \bysame, {\em Recent developments on noncommutative motives}. Available at arXiv:1611.05439. To appear in Contemporary Mathematics, AMS.

\bibitem{Gysin} G.~Tabuada and M.~Van den Bergh, {\em The Gysin triangle via localization and $\mathbb{A}^1$-homotopy invariance}. Transactions of the American Mathematical Society {\bf 370} (2018), no.~1, 421--446.

\bibitem{Azumaya} \bysame, {\em Noncommutative motives of Azumaya algebras}. J. Inst. Math. Jussieu {\bf 14} (2015), no. 2, 379--403.

\bibitem{Orbifold} \bysame, {\em Additive invariants of orbifolds}. Available at arXiv:1612.03162. To appear in Geometry and Topology.

\bibitem{Thomas} R.~Thomas, {\em Notes on HPD}. Available at arXiv:1512.08985. To appear in Proceedings of the 2015 AMS Summer Institute, Salt Lake City.

\end{thebibliography}
\end{document}

\end{proof}